\documentclass[12pt]{amsart}
\usepackage{amssymb,bbold}
\usepackage[all]{xy}

\theoremstyle{plain}
\newtheorem{theorem}{Theorem}
\newtheorem{corollary}[theorem]{Corollary}
\newtheorem{lemma}[theorem]{Lemma}
\newtheorem{proposition}[theorem]{Proposition}

\theoremstyle{definition}
\newtheorem{definition}[theorem]{Definition}

\begin{document}
\baselineskip 18pt

\title[The Banach-Saks property from a locally solid vector lattice]
      {The Banach-Saks property from a locally solid vector lattice point of view}

\author[O.~Zabeti]{Omid Zabeti}


\address[O.~Zabeti]
  {Department of Mathematics, Faculty of Mathematics, Statistics, and Computer science,
   University of Sistan and Baluchestan, Zahedan,
   P.O. Box 98135-674. Iran}
\email{o.zabeti@gmail.com}


\begin{abstract}
The aim of this note is to consider different notions for the Banach-Saks property in locally solid vector lattices  as an extension for the known concepts of the Banach-Saks property in Banach lattices. We investigate relations between them; in particular, we shall characterize spaces in which, these notions agree.
\end{abstract}

\date{\today}

\maketitle
\section{motivation and preliminaries}
Let us start with some motivation. Let $E$ be a Banach lattice. It is known that there are several non-equivalent notions related to the Banach-Saks property on $E$ which have many considerable connections together. In particular, we are able to characterize order continuity and also reflexivity of $E$ in terms of these relations ( see \cite{GTX, Z2} for more information).

 Although, Banach lattices are the most important part of the category of all locally solid vector lattices, but the general theory of locally solid vector lattices has practical applications in economics, risk measures, etc ( see \cite[Chapter 9]{AB1}). Furthermore, the set of all examples in this category is  much more than  the examples in Banach lattice theory so that in this case, we have a wider range of spaces. All of these reasons, motivate us to investigate several known notions which are defined in Banach lattices in the category of all locally solid vector lattices; one principal notion is the Banach-Saks property.

In this paper, we consider the Banach-Saks property ( in four non-equivalent ways) in the category of all locally solid vector lattices as an extension for the corresponding properties in the category of all Banach lattices. Recently, some different notions for the Banach-Saks property in Banach lattices have been considered in \cite{Z2}. Our attempt in this note is to investigate and possibly extend the related results in that paper to the category of all locally solid vector lattices.

 Now, we recall some preliminaries we need in the sequel.
A vector lattice $X$ is called {\em order complete} ( {\em $\sigma$-order complete}) if every non-empty bounded above subset ( bounded above countable subset) of $X$ has a supremum. A set $S\subseteq X$ is called a {\em solid} set if $x\in X$, $y\in S$ and $|x|\leq |y|$ imply that $x\in S$. Also, recall that a linear topology $\tau$ on a vector lattice $X$ is referred to as  {\em locally solid} if it has a local basis at zero consisting of solid sets.
Suppose $X$ is a locally solid vector lattice. A net $(x_{\alpha})\subseteq X$ is said to be unbounded convergent to $x\in X$ ( in brief, $x_{\alpha}\xrightarrow{u\tau}x$) if for each $u\in X_{+}$, $|x_{\alpha}-x|\wedge u\xrightarrow{\tau}0$; suppose $\tau_1=\|.\|$ and $\tau_2$= the absolute weak topology ( $|\sigma|(X,X')$, in notation), we write $x_{\alpha}\xrightarrow{un}x$ and $x_{\alpha}\xrightarrow{uaw}x$, respectively; for more details, see \cite{DOT,T,Z}. Observe that a locally solid vector lattice $(X,\tau)$ is said to have the {\em Levi property} if every $\tau$-bounded upward directed set in $X_{+}$ has a supremum. Moreover, recall that a locally solid vector lattice $(X,\tau)$ possesses the {\em Lebesgue property} if for every net $(u_{\alpha})$ in $X$, $u_{\alpha}\downarrow 0$ implies that $u_{\alpha}\xrightarrow{\tau}0$. Finally, observe that $(X,\tau)$ has the {\em pre-Lebesgue property} if $0\leq u_{\alpha}\uparrow\leq u$ implies that the net $(u_{\alpha})$ is $\tau$-Cauchy.
For undefined terminology and related notions, see \cite{AA,AB1,AB}. All locally solid vector lattices in this note are assumed to be Hausdorff.

\section{main result}
First, we present the main definition of the paper: The Banach-Saks property with a topological flavor; for recent progress in Banach lattices, see \cite{GTX,Z2}.

\begin{definition}
Suppose $(X,\tau)$ is a locally solid vector lattice. $X$ is said to have
\begin{itemize}
		\item[\em (i)] {The Banach-Saks property ( {\bf LSBSP}, for short) if every  $\tau$-bounded sequence $(x_n)\subseteq X$ has a subsequence $(x_{n_k})$ whose Ces\`{a}ro means is $\tau$-convergent; that is $(\frac{1}{n}\Sigma_{k=1}^{n}x_{n_k})$ is $\tau$-convergent.}
\item[\em (ii)]
	{ The strong unbounded Banach-Saks property ( {\bf LSSUBSP}, in notation) if every  $\tau$-bounded $u\tau$-null sequence $(x_n)\subseteq X$ has a subsequence $(x_{n_k})$ whose Ces\`{a}ro means is $\tau$-convergent.}
\item[\em (iii)]
	{ The unbounded Banach-Saks property ( {\bf LSUBSP}, in notation) if every  $\tau$-bounded $u|\sigma|(X,X')$-null sequence $(x_n)\subseteq X$ has a subsequence $(x_{n_k})$ whose Ces\`{a}ro means is $\tau$-convergent.}
		\item[\em (iv)] {The disjoint Banach-Saks property ( {\bf LSDBSP}, in brief) if every  $\tau$-bounded disjoint sequence $(x_n)\subseteq X$ has a subsequence $(x_{n_k})$ whose Ces\`{a}ro means is $\tau$-convergent.}
		
			\end{itemize}
\end{definition}
It is clear that {\bf LSBSP} implies {\bf LSSUBSP}, {\bf LSUBSP} and {\bf LSDBSP}.
Now, we extend \cite[Theorem 3.2]{DOT} to the category of all Fr$\acute{e}$chet spaces. The proof is essentially the same; just, it is enough to replace norm with the compatible translation-invariant metric.  Recall that a metrizable locally solid vector lattice $X$ is a Fr$\acute{e}$chet space if it is complete.
\begin{lemma}\label{701}
Suppose $(X, \tau)$ is a Fr$\acute{e}$chet space and $(x_{\alpha})$ is a $u\tau$-null net in $X$. Then there exists an increasing sequence of indices $(\alpha_k)$ and a disjoint sequence $(d_k)$ such that $x_{\alpha_k}-d_k\xrightarrow{\tau} 0$.
\end{lemma}
Note that metrizability is essential in Lemma \ref{701} and can not be removed. Consider $X=\ell_1$ with the absolute weak topology. Put $u_n=(\underbrace{0,\ldots,0}_{n-times},\underbrace{n,\ldots,n}_{n-times},0,\ldots)$. It is easy to see that $u_n\rightarrow0$ in the unbounded absolute weak topology ( $uaw$-topology); for more details, see \cite{Z}. Now, suppose $(d_n)$ is any disjoint sequence in $X$. It is easy to see that the sequence $(x_n-d_n)$ has at least one component with value $n$. This means that $x_n-d_n\nrightarrow 0$ in the weak topology so that in the absolute weak topology.

In this step, let us recall definition of the $AM$-property; for more details, see \cite{Z1}.
Suppose $(X,\tau)$ is a locally solid vector lattice. We say that $X$ has the  $AM$-property provided that for every bounded set $B\subseteq X$, $B^{\vee}$ is also bounded with the same scalars; namely, given a zero neighborhood $V$ and any positive scalar $\alpha$ with $B\subseteq \alpha V$, we also have  $B^{\vee}\subseteq \alpha V$. Note that by $B^{\vee}$, we mean the set of all finite suprema of elements of $B$. Observe that $B^{\vee}$ can be considered as an increasing net in $X$.
\begin{lemma}\label{5001}
Suppose $X$ is a locally solid vector lattice with the $AM$-property and $U$ is an arbitrary solid zero neighborhood in $X$. Then, for each $m\in \Bbb N$, $U\vee\ldots\vee U=U$, in which $U$ is appeared $m$-times.
\end{lemma}
\begin{proof}
Suppose $x\in U$; without loss of generality, we may assume that $x\geq 0$ so that $x=x\vee 0\ldots\vee 0$. Therefore, $x\in U\vee\ldots\vee U$. For the other direction, suppose that $x_1,\ldots,x_m\ \in U$. Put $F=\{x_1,\ldots,x_m\}$; $F$ is bounded so that by definition of the $AM$-property, $x_1\vee\ldots\vee x_m \in U$.
\end{proof}
\begin{lemma}\label{702}
Suppose $X$ is a locally solid vector lattice with the $AM$-property. Then $X$ possesses the {\bf LSDBSP}.
\end{lemma}
\begin{proof}
Suppose $(x_{n})$ is a $\tau$-bounded disjoint sequence in $X$ and $V$ is an arbitrary solid zero neighborhood in $X$. There exists a positive scalar $\gamma$ with $(x_n)\subseteq \gamma V$. Assume that $B$ is the set of all finite suprema of elements of $(x_n)$. By the $AM$-property, $B$ is also bounded and $B\subseteq \gamma V$. Therefore, for any subsequence $(x_{n_k})$ of $(x_n)$, we have
\[\frac{1}{n}\Sigma_{k=1}^{n}x_{\alpha_k}=\frac{1}{n}\bigvee_{k=1}^{n}x_{\alpha_k}\in \frac{\gamma}{n}V\subseteq V,\]
for sufficiently large $n$. This would complete the proof.
\end{proof}
\begin{proposition}\label{103}
Suppose $X$ is a $\sigma$-order complete locally solid vector lattice. Then {\bf LSUBSP} in $X$ implies {\bf LSDBSP}.
\end{proposition}
\begin{proof}
First, assume that $X$ possesses the pre-Lebesgue property. Suppose $(x_n)$ is a $\tau$-bounded disjoint sequence in $X$. By \cite[Theorem 4.2]{T}, $x_{n}\xrightarrow{u\tau}0$ so that $x_{n}\xrightarrow{u|\sigma|(X,X')}0$. Therefore, there exists an increasing sequence $(n_k)$ of indices such that the sequence $(\frac{1}{n}\Sigma_{k=1}^{n}x_{n_k})$ is $\tau$-convergent. Now, suppose $X$ does not have the pre-Lebesgue property. By \cite[Theorem 3.28]{AB1}, $\ell_{\infty}$ is lattice embeddable in $X$ so that $X$ can not have {\bf LSUBSP}. This would complete the proof.
\end{proof}
For the converse, we have the following.
\begin{theorem}\label{703}
Suppose $(X,\tau)$ is a $\sigma$-order complete locally convex-solid vector lattice which is also a  Fr$\acute{e}$chet space. Then {\bf LSDBSP} results {\bf LSUBSP} in $X$ if and only if $X$ does not contain a lattice copy of $\ell_{\infty}$.
\end{theorem}
\begin{proof}
 The direct implication is trivial since $\ell_{\infty}$ does not have {\bf LSUBSP}. For the other side, assume that $(x_n)$ is a $\tau$-bounded $u|\sigma|(X,X')$-null sequence in $X$. This means that for each $u\in X_{+}$, $x_{n}\wedge u\rightarrow 0$ absolutely weakly.
 By \cite[Theorem 6.17 and Theorem 3.28]{AB1}, $x_n\wedge u \xrightarrow{\tau}0$.  By Lemma \ref{701}, there exist an increasing sequence $(n_k)$ of indices and a disjoint sequence $(d_k)$ in $X$ such that $x_{n_k}-d_k\xrightarrow{\tau}0$. By passing to a subsequence, we may assume that $\lim \frac{1}{m}\Sigma_{k=1}^{m}d_k\rightarrow 0$. Now, the result follows from the following inequality; observe that when a sequence in a topological vector space is null, then so is its Ces\`{a}ro means.
\[\frac{1}{m}\Sigma_{k=1}^{m}x_{n_k}-\frac{1}{m}\Sigma_{k=1}^{m}d_k= \frac{1}{m}\Sigma_{k=1}^{m}(x_{n_k}-d_k)\rightarrow 0.\]
\end{proof}
It is clear that {\bf LSUBSP} implies {\bf LSSUBSP}. For the converse, we have the following characterization.
\begin{theorem}
Suppose $(X,\tau)$ is a $\sigma$-order complete  locally convex-solid vector lattice. Then {\bf LSSUBSP} in $X$ implies  {\bf LSUBSP} if and only if $X$ does not contain a lattice copy of $\ell_{\infty}$.
\end{theorem}
\begin{proof}
Note that $\ell_{\infty}$ possesses {\bf LSSUBSP} by \cite[Theorem 2.3]{KMT} but it fails to have {\bf LSUBSP}; in this case, it should have the weak Banach-Saks property which is not possible.
Now, suppose $X$ does not contain a lattice copy of $\ell_{\infty}$. By \cite[Theorem 3.28]{AB1}, $X$ possesses the pre-Lebesgue property. Now, suppose $(x_n)\subseteq X$ is a bounded sequence which is $u|\sigma|(X,X')$-null. Therefore, by \cite[Theorem 6.17]{AB1}, $(x_n)$ is also $u\tau$-null. By the assumption, there exists a subsequence $(x_{n_k})$ of $(x_n)$ whose Ces\`{a}ro means is convergent, as claimed.
\end{proof}

\begin{proposition}
Suppose a topologically complete locally convex-solid vector lattice $(X,\tau)$ possesses {\bf LSBSP}. Then $(X,\tau)$ possesses the Lebesgue and Levi properties.
\end{proposition}
\begin{proof}
Suppose not. By \cite[Theorem 1]{W}, $X$ contains a lattice copy of $c_0$. This means that $X$ fails to have {\bf LSBSP}.
\end{proof}
\begin{proposition}\label{6000}
Suppose $(X,\tau)$ is an atomic locally solid vector lattice which possesses the Lebesgue and Levi properties. Then {\bf LSSUBSP} in $X$ implies {\bf LSBSP}.
\end{proposition}
\begin{proof}
Suppose $(x_n)$ is a $\tau$-bounded sequence in $X$. By \cite[Theorem 6]{DEM}, there exists a subsequence $(x_{n_k})$ of $(x_n)$ which is $u\tau$-convergent. By the assumption, there is a sequence $(n_{k_{k'}})$ of indices such that the sequence $(x_{n_{k_{k'}}})$ has a convergent Ces\`{a}ro means. This would complete the proof.
\end{proof}
\begin{corollary}
Suppose $(X,\tau)$ is an atomic locally solid vector lattice which possesses the Lebesgue and Levi properties. Then {\bf LSUBSP} in $X$ implies {\bf LSBSP}.
\end{corollary}
Observe that atomic assumption is necessary in Proposition \ref{6000} and can not be omitted; consider $X=L^1[0,1]$ with the norm topology. It possesses the Lebesgue and Levi properties but it is not atomic. Suppose $(f_n)$ is a bounded sequence in $X$ which is $un$-convergent to zero; by \cite[Corollary 4.2]{DOT}, it is convergent in measure. By \cite[Theorem 1.82]{AA}, there exists a subsequence $(f_{{n}_k})$ of $(f_{n})$ which is also pointwise convergent. By \cite[Corollary 1.86]{AA}, it is uniformly relatively convergent to zero. This implies that the Ces\`{a}ro means of this subsequence is convergent. So, $X$ possesses {\bf LSSUBSP}; nevertheless, it fails to have {\bf LSBSP}, certainly.


\begin{thebibliography}{1}
\bibitem{AA} Y. Abramovich and C.D. Aliprantis, An invitation to Operator theory,
Vol. 50. Providence, RI: American Mathematical Society, 2002.
\bibitem{AB1} C. D. Aliprantis and O. Burkinshaw, Locally solid Riesz spaces with applications to economics, Mathematical Surveys and Monographs, 105, American Mathematical Society, Providence, 2003.
\bibitem{AB} C.D. Aliprantis and O. Burkinshaw, Positive operators,
Springer, 2006.
 \bibitem{DEM} Y. Dabboorasad, E. Y. Emelyanov, and M. A. A. Marabeh, {\em $um$-topology in multi-normed vector lattices}, Positivity,  {\bf 22(2)} (2018), pp. 653--667.
\bibitem{DOT} Y. Deng, M O'Brien, and V. G. Troitsky, {\em Unbounded norm convergence in Banach lattices,} Positivity, {\bf 21(3)} (2017), pp. 963--974.
  \bibitem{KMT} M. Kandi\'{c}, M. A. A. Marabeh, and V. G. Troitsky, {\em Unbounded norm topology in Banach lattices}, J.
Math. Anal. Appl., {\bf 451 (1)}(2017), pp. 259--279.
\bibitem{GTX} N. Gao, V. G. Troitsky, and F. Xanthos, {\em Uo-convergence and its applications to Cesàro means in Banach lattices,}  Israel J. Math, {\bf 220} (2017), pp. 649--689.
    \bibitem{T} M. A. Taylor, {\em Unbounded topologies and uo-convegence in locally solid vector lattices}, J.
Math. Anal. Appl., {\bf 472 (1)}(2019), pp. 981--1000.
\bibitem{W} W. Wnuk, {\em  Locally solid Riesz spaces not containing $c_0$}, Bull. Polish Acad. Sci. Math., {\bf 36(1–2)}, pp. 51-–56.
(1988)
\bibitem{Z1} O. Zabeti, {\em AM-Spaces from a locally solid vector lattice point of view
with applications}, Bulletin. Iran. Math. Society, to appear, doi: 10.1007/s41980-020-00458-7.
\bibitem{Z} O. Zabeti, {\em Unbounded absolute weak convergence in Banach lattices}, Positivity,  {\bf 22(1)} (2018), pp. 501--505.
\bibitem{Z2} O. Zabeti, {\em Unbounded continuous operators and unbounded Banach-Saks property in Banach lattices}, Preprint, arXiv: 2007.05734v3 [math.FA].
\end{thebibliography}
\end{document}